\documentclass[10pt,article]{amsart}
\usepackage{amssymb}

\theoremstyle{plain}
\newtheorem{theorem}{Theorem}
\newtheorem{lemma}{Lemma}
\newtheorem{example}{Example}
\newtheorem{prop}{Proposition}

\newtheorem{question}{Question}

\newcommand\diag{\operatorname{diag}}

\newcommand\tr{\operatorname{tr}}
\DeclareMathOperator*{\Llim}{Llim}
\newcommand\rank{\operatorname{rank}}

\newcommand{\C}{\mathbb{C}}

\newcommand{\N}{\mathbb{N}}
\newcommand{\Z}{\mathbb{Z}}
\newcommand{\R}{\mathbb{R}}

\newcommand{\irB}{\mathcal{B}}

\newcommand{\irH}{\mathcal{H}}

\newcommand{\irM}{\mathcal{M}}

\newcommand{\irP}{\mathcal{P}}

\newcommand{\irU}{\mathcal{U}}
\newcommand{\irW}{\mathcal{W}}
\newcommand\PWB{\irP\irW\irB}

\begin{document}

\title{Characterisation of Ces\`aro and $L$-Asymptotic Limits of Matrices}

\author{Gy\"orgy P\'al Geh\'er}

\begin{abstract}
The main goal of this paper is to characterise all the possible Ces\`aro and $L$-asymptotic limits of powerbounded, complex matrices. The investigation of the $L$-asymptotic limit of a powerbounded operator goes back to Sz.-Nagy and it shows how the orbit of a vector behaves with respect to the powers. It turns out that the two types of asymptotic limits coincide for every powerbounded matrix and a special case is connected to the description of the products $SS^*$ where $S$ runs through those invertible matrices which have unit columnvectors. We also show that for any powerbounded operator acting on an arbitrary complex Hilbert space the norm of the $L$-asymptotic limit is greater than or equal to 1, unless it is zero; moreover, the same is true for the Ces\`aro asymptotic limit of a not necessarily powerbounded operator, if it exists.
\end{abstract}

\maketitle

\let\thefootnote\relax\footnote{
This research was realized in the frames of T\'AMOP 4.2.4. A/2-11-1-2012-0001 ''National Excellence Program - Elaborating and operating an inland student and researcher personal support system''. The project was subsidized by the European Union and co-financed by the European Social Fund.\\
The author was also supported by the "Lend\"ulet" Program (LP2012-46/2012) of the Hungarian Academy of Sciences.}

\section{Introduction}

Let us endow $\C^d$ ($d\in\N := \{1,2,\dots\}$) with the usual inner-product $\langle\cdot,\cdot\rangle$. Both of the symbols $\irM_d$ and $\C^{d\times d}$ will stand for the set of all $d\times d$ matrices. We will denote the elements of the usual standard basis by $e_1,\dots e_d\in\C^d$. Throughout this paper the notion of matrices and operators on $\C^d$ will be identified in the following natural ways: on the one hand, any operator $T$ will be associated with the matrix $\big(\langle T e_j, e_i\rangle\big)_{i,j=1}^n$; on the other hand, if we have a $d\times d$ complex matrix then the corresponding operator will be precisely that one for which the image of $e_j$ is the $j$th columnvector of our matrix ($1\leq j\leq d$).

Since our motivation, which goes back to B. Sz.-Nagy and will be explained in the next paragraphs, comes from Hilbert space operator theory let us now define some necessary definitions. Let $\irH$ be a (not necessarily finite dimensional) complex Hilbert space with the inner-product $\langle\cdot,\cdot\rangle$, and $\irB(\irH)$ the set of all bounded linear operators acting on it. An operator $A\in\irB(\irH)$ is said to be \emph{positive}, in notation: $A\in\irB_+(\irH)$, if $\langle Ah,h\rangle \geq 0$ is satisfied for each vector $h\in\irH$. In the case of matrices we also call $A$ a \emph{positive semi-definite matrix}. If the positive semi-definite matrix $A$ is also invertible then we refer to it as a \emph{positive definite matrix}. The set of all $d\times d$ unitaries and invertible matrices will be denoted by $\irU_d$ and $GL_d$, respectively. For any $T\in\irB(\irH)$ the quantity $\|T\|$ will stand for the operator norm of $T$.

Let us consider a contraction $T\in\irB(\irH)$ (i. e. $\|T\|\leq 1$) and the sequence $\{T^{*n}T^n\}_{n=1}^\infty$ of positive operators which is trivially decreasing. Therefore there exists a unique limit
\[ A_T = A := \lim_{n\to\infty} T^{*n}T^n \]
in the strong operator-topology (the topology of point-wise convergence). Clearly, $A_T$ is a positive operator, and it will be called the \emph{asymptotic limit} of the contraction $T$ (see \cite{Ku1}, \cite{Ku2}, \cite{NFBK} for more details). The asymptotic limit tells us how the orbit of a vector behaves, namely the following is true 
\[ \lim_{n\to\infty} \|T^n h\| = \lim_{n\to\infty}\sqrt{\langle T^n h,T^n h\rangle} = \sqrt{\langle A_T h,h\rangle} = \big\|\sqrt{A_T}h\big\| = \big\|\sqrt{A_T}Th\big\|. \]

In \cite{SzN} Sz.-Nagy considered powerbounded operators and defined a generalization of the context above. An operator $T\in\irB(\irH)$ is said to be \emph{powerbounded}, in notation $T\in\PWB(\irH)$ (or $T\in\PWB_d$ if $\irH = \C^d$) if there exists a bound $M > 0$ such that $\|T^n\| < M$ holds for each $n\in\N$. In order to define Sz.-Nagy's asymptotic limit for all powerbounded operators we need the notion of Banach limits. The Banach space of bounded complex sequences is denoted by $\ell^\infty$. We say that the bounded linear functional 
\[ L\colon \ell^\infty \to \C, \quad \underline{x} \mapsto \Llim_{n\to\infty} x_n \]
is a \emph{Banach limit} if the next four points are satisfied:
\begin{itemize}
\item $\|L\| = 1$,
\item we have $\Llim_{n\to\infty} x_n = \lim_{n\to\infty} x_n$ for every convergent sequence,
\item $L$ is positive (i. e. if $x_n\geq 0$ for all $n\in\N$, then $\Llim_{n\to\infty} x_n \geq 0$) and
\item $L$ is shift-invariant (i. e. $\Llim_{n\to\infty} x_n = \Llim_{n\to\infty} x_{n+1}$).
\end{itemize}
Note that a Banach limit is never multiplicative. Now let us take an arbitrary $T\in\PWB(\irH)$ and fix a Banach limit $L$. We consider the following sesqui-linear form
\[ \irH\times\irH \to \C, \quad (x,y)\mapsto \Llim_{n\to\infty}\langle T^nx,T^ny\rangle \]
which is obviously bounded and hence there exists a unique positive operator $A_{T,L}\in\irB_+(\irH)$ which represents this form, i. e.
\[ \Llim_{n\to\infty}\langle T^nx,T^ny\rangle = \langle A_{T,L}x,y\rangle \qquad (\forall \; x,y\in\irH). \]
This $A_{T,L}$ (which obviously depends on the choice of $L$, see the last section) will be called here the \emph{$L$-asymptotic limit} of $T$ (see e. g. \cite{Ke_isom_as}, \cite{NFBK} or \cite{SzN} for more details). If $T$ is a $d\times d$ matrix, we obtain $A_{T,L}$ by taking the entry-wise $L$-limit of the matrix-sequence $\{T^{*n}T^n\}_{n=1}^\infty$. Clearly, we have
\[ \Llim_{n\to\infty} \|T^nx\|^2 = \big\|\sqrt{A_{T,L}}x\big\|^2 = \big\|\sqrt{A_{T,L}}Tx\big\|^2, \]
but this is usually not true if we erase the squares. 

The notion of asymptotic limits and its generalizations play an important role in the hyper-invariant subspace problem (see e. g. \cite{BK} \cite{Ca}, \cite{Ke_gen_Toep} \cite{Ke_isom_as}, \cite{Ke_reg_norm_seq}, \cite{KT}, \cite{NFBK} or \cite{SS}). For instance it can be verified quite easily that for any $T\in\PWB(\irH)$ the subspace $\irH_0(T) = \irH_0 := \ker(A_{T,L}) = \{ x\in\irH \colon T^n x\to 0 \}$ is hyper-invariant for $T$. This subspace is called the \emph{stable subspace} of $T$. It can be shown that $\irH_0 = \ker(A_{T,L}) = \ker(A_T)$ is fulfilled for contractions.

Using the $L$-asymptotic limit, Sz.-Nagy was managed to prove in \cite{SzN} that any operator $T\in\irB(\irH)$ is similar to a unitary operator exactly when both $T$ and $T^{-1}$ are powerbounded (see a re-phrasing later in this section). In \cite{Du}, E. Durszt used the operator $A_T$ to prove a generalization of the Rota model (see \cite{Ro}) for completely non-unitary contractions (see also \cite{Ku1}). In \cite{NFBK} the reader can find a construction of the minimal unitary dilation of a contraction which uses the asymptotic limit. C. S. Kubrusly pointed out in \cite{Ku2} that it is important to find multifarious characterisations for the case when a contraction has projective asymptotic limit. In the case when the asymptotic limit is injective, one can conclude some properties of $T$ from the properties of a corresponding isometry which is defined in a natural way (see e. g. \cite{Ge_tree}). 

In \cite{Ge} the present author characterised all of the possible asymptotic limits of Hilbert space contractions. It was noted there that any contractive matrix $T\in\C^{d\times d}$ has an idempotent asymptotic limit. In what follows, we are interested in the characterization of all the possible $L$-asymptotic limits of powerbounded matrices. We note that the proofs in \cite{Ge} used infinite dimensional techniques, so in order to proceed we need other ideas.

Before stating our main theorems we need the definition of the Ces\`aro asymptotic limit which is an other possible generalization of the notion of asymptotic limit for contractions. We call $A_{T,C}\in\irB(\irH)$ the \emph{Ces\`aro asymptotic limit} of $T\in\irB(\irH)$ if $\frac{1}{n}\sum_{j=1}^n T^{*j}T^j \to A_{T,C}$ holds in the weak operator-topology (the topology of point-wise weak convergence). In the matrix case we can take entry-wise convergence of the Ces\`aro means instead. Obviously, if $T$ is a contraction then $A_{T,C}$ exists and coincides with $A_T$. But generally for a $T\in\PWB(\irH)$ the Ces\`aro limit does not exist always. We also note that even the existence of the Ces\`aro asymptotic limit usually does not imply the powerboundedness of $T$. Counterexamples will be provided in the last section. However, as we will see from Theorem \ref{C^d_Cer_mean_thm}, these are equivalent conditions for matrices. Trivially, if the Ces\`aro asymptotic limit of $T\in\irB(\irH)$ exists, then
\begin{equation}\label{Cesaro_as_beh_eq}
\big\|\sqrt{A_{T,C}}h\big\| = \sqrt{\langle A_{T,C} h,h\rangle} = \lim_{n\to\infty}\sqrt{\frac{1}{n}\sum_{j=1}^n \|T^j h\|^2}
\end{equation}
holds for any $h\in\irH$. We note that 
\begin{equation}\label{Cesaro_iter_eq} 
\begin{gathered}
\big\|\sqrt{A_{T,C}}T h\big\| = \lim_{n\to\infty}\sqrt{\frac{1}{n}\sum_{j=2}^{n+1} \|T^j h\|^2}\\
= \lim_{n\to\infty}\sqrt{\frac{n+1}{n}\cdot\Bigg[\frac{1}{n+1}\sum_{j=1}^{n+1} \|T^j h\|^2\Bigg] - \frac{1}{n}\|h\|^2} = \big\|\sqrt{A_{T,C}}h\big\|. 
\end{gathered}
\end{equation}
If for the powerbounded $T$ the Ces\`aro asymptotic limit $A_{T,C}$ exists then $\irH_0 = \ker A_{T,C}$ holds. This can be verified rather easily.

Obviously for each operator $T\in\irB(\irH)$ and unitary operator $U\in\irB(\irH)$ we have one of the following assertions
\begin{equation}\label{unitary_eqv_eq}
A_{UTU^*} = U A_T U^*, \quad A_{UTU^*,L} = U A_{T,L} U^*, \quad A_{UTU^*,C} = U A_{T,C} U^* 
\end{equation}
if at least one side of the corresponding equation makes sense.

The main aim of this work is to provide the characterisation of all possible $L$-asymptotic limits in finite dimension. The first theorem says that the $L$ asymptotic limit of a powerbounded matrix coincides with the Ces\`aro asymptotic limit and it will be proved in the next section. 

\begin{theorem}\label{C^d_B_lim_Cer_thm}
For every matrix $T\in \PWB(\C^d)$ $A_{T,C} = A_{T,L}$ holds for all Banach limits $L$.
\end{theorem}

Note that in the infinite dimensional case it can happen for a powerbounded operator $T$ that the Ces\`aro asymptotic limit exists and there is a Banach limit $L$ such that $A_{T,L} \neq A_{T,C}$ (see the last section). The theorem in \cite{SzN} can be re-phrased as follows: an operator $T\in\irB(\irH)$ is similar to a unitary operator if and only if $T$ is powerbounded and both $A_{T,L}$ and $A_{T^*,L}$ are invertible. In case when $A_{T,L}$ and $A_{T^*,L}$ are only injective, B. Sz.-Nagy and C. Foias called the powerbounded operator $T$ of class $C_{11}$, $T\in C_{11}(\irH)$ in notation. Clearly, in $\C^d$ these previous notions are the same. We will first characterise those Ces\`aro (or $L$-) asymptotic limits that arise from a $C_{11}$ powerbounded matrix or equivalently, that arise from an operator which is similar to a unitary one.

\begin{theorem}[Characterisation in the $C_{11}$ case]\label{C_11_char_thm}
The following statements are equivalent for a positive definite $A\in\irM_d$
\begin{itemize}
\item[\textup{(i)}] $A$ is the Ces\`aro asymptotic limit of a $T \in C_{11}(\C^d)$,
\item[\textup{(ii)}] $A$ is the $L$-asymptotic limit of a $T \in C_{11}(\C^d)$,
\item[\textup{(iii)}] if the eigenvalues of $A$ are $t_1,\dots,t_d > 0$, each of them is counted according to their multiplicities, then
\begin{equation}\label{C11_eq}
\frac{1}{t_1}+\dots+\frac{1}{t_d} = d
\end{equation}
holds,
\item[\textup{(iv)}] there is an $S\in GL_d$ with unit columnvectors such that 
\[ A = S^{*-1}S^{-1} = (SS^*)^{-1}. \]
\end{itemize}
\end{theorem}

The above theorem will be proved in Section \ref{char_sec}. After that in the same section we will be able to deal with the non-$C_{11}$ case. In the proof of that case we will use the $C_{11}$ case and a block-diagonalization of a special type of block matrices. We call a $T\in\PWB_d$ \emph{$l$-stable} ($0\leq l\leq d$) if $\dim\irH_0 = l$. We will see later that the $C_{11}$ class powerbounded matrices are exactly the 0-stable matrices and that for a $T\in\PWB_d$ the matrices $T$ and $T^*$ are simultaneously $l$-stable. In the forthcoming theorem the symbol $I_l$ stands for the $l\times l$ identity matrix, $0_{k}\in\irM_{k}$ is the zero matrix and $\oplus$ denotes the orthogonal sum.

\begin{theorem}[Characterisation of the non-$C_{11}$ case]\label{non-C11_char_thm}
The following four points are equivalent for a non-invertible $A\in\irB_+(\C^d)$ and $1\leq l< d$:
\begin{itemize}
\item[\textup{(i)}] there exists an $l$-stable $T\in\PWB(\C^d)$ such that $A_{T,C} = A$,
\item[\textup{(ii)}] there exists an $l$-stable $T\in\PWB(\C^d)$ such that $A_{T,L} = A$,
\item[\textup{(iii)}] let $k = d-l$, if $t_1, \dots t_k$ denotes the non-zero eigenvalues of $A$ counted with their multiplicities, then 
\[\frac{1}{t_1}+\dots+\frac{1}{t_k}\leq k,\]
\item[\textup{(iv)}] there exists such an $S\in GL_d$ that has unit columnvectors and 
\[ A = S^{*-1}(I_l\oplus 0_{k})S^{-1}. \] 
\end{itemize}
\end{theorem}

One could ask whether is there any connection between the Ces\`aro asymptotic limit of a matrix and the Ces\`aro asymptotic limit of its adjoint. If $T\in\irM_d$ is contractive then the asymptotic limits satisfies the equality $A_{T^*} = A_T$ (see \cite{Ku1}) moreover, they are the projections onto the subspace $\irH_0(T)^\perp = \irH_0(T^*)^\perp$. In the case of powerbounded matrices usually the subspaces $\irH_0(T)$ and $\irH_0(T^*)$ are different and hence $A_{T^*}$ and $A_T$ differ, too. However, we provide the next connection for $C_{11}$ class $2\times 2$ powerbounded matrices which will be proved in Section \ref{char_sec}.

\begin{theorem}\label{2x2_regularity_thm}
For each $T\in C_{11}(\C^2)$ the harmonic mean of the Ces\`aro asymptotic limits $A_{T,C}$ and $A_{T^*,C}$ is exactly the identity $I$, i. e. 
\begin{equation}\label{2_dim_nice_eq}
A_{T,C}^{-1} + A_{T^*,C}^{-1} = 2I_2. 
\end{equation}
\end{theorem}

In Section \ref{norm_of_as_lim_chap} we will deal with operators acting on an arbitrary space and tell some properties of the Ces\`aro asymptotic limits and $L$-asymptotic limits. In \cite{Ge} the author proved that the asymptotic limit of a Hilbert space contraction has norm 1 or 0. The next two results tell us a similar property of Ces\`aro and $L$-asymptotic limits.

\begin{theorem}\label{A_TC_norm}
Assume $T$ is a (not necessarily powerbounded) operator for which $A_{T,C}$ exists and it is not zero. Then the inequality 
\[ \|A_{T,C}\| \geq 1 \] 
is fulfilled.
\end{theorem}

\begin{theorem}\label{A_TL_norm}
Suppose $L$ is a fixed Banach limit and $T$ is a powerbounded operator for which $A_{T,L} \neq 0$ holds. Then the inequality 
\[ \|A_{T,L}\| \geq 1 \] 
is satisfied.
\end{theorem}

\section{The Ces\`aro and $L$-asymptotic limit of a matrix coincide}\label{C_L-as_lim_sec}

This section is devoted to prove Theorem \ref{C^d_B_lim_Cer_thm} but before that we need some auxiliary results. First, we describe the Jordan decomposition of a powerbounded matrix. For an operator $T\in\irB(\irH)$ the quantity $r(T)$ stands for the spectral radius of $T$. Note that for any operator $B\in\irB(\irH)$ the equation $\lim_{n\to\infty}\|B^n\| = 0$ holds exactly when $r(B)<1$. The verification of this is quite easy for a matrix by the Jordan decomposition theorem but it is also valid in the operator case (see \cite[Proposition 0.4]{Ku2}). The symbol $\diag(\dots)$ expresses a diagonal matrix.

\begin{prop}[Jordan decomposition of powerbounded matrices]\label{J_dec_pwb_mx_prop}
Suppose $T\in\PWB_d$ and let us consider its Jordan decomposition: $SJS^{-1}$. Then we have
\[ J = U\oplus B \] 
with a unitary $U = \diag(\lambda_1,\dots\lambda_k)\in\irM_k$ ($k\in\Z_+ = \N\cup\{0\}$) and a $B\in\irM_{d-k}$ for which $r(B) < 1$.

Conversely, if $J$ has the previous form, then necessarily $T$ is powerbounded.
\end{prop}

\begin{proof} For the first assertion, let us consider the block-decomposition of $J$: 
\[ J = J_{1}\oplus \dots \oplus J_k \] 
where $J_j$ is a $\lambda_j$-Jordan block. It is clear that powerboundedness is preserved by similarity. Since
\begin{equation}\label{Jordan-block_power_eq}
J_j^n = \left( \begin{matrix}
\lambda_j & 1 & 0 & \dots & 0 \\
0 & \lambda_j & 1 & \dots & 0 \\
0 & 0 & \lambda_j & \ddots & 0 \\
\vdots & \vdots & \ddots & \ddots & \vdots \\
0 & 0 & 0 & \dots & \lambda_j
\end{matrix} \right)^n 
=
\left( \begin{matrix}
\lambda_j^n & n\lambda_j^{n-1} & \binom{n}{2}\lambda_j^{n-2} & \ddots & \ddots \\
0 & \lambda_j^n & n\lambda_j^{n+1} & \ddots & \ddots \\
0 & 0 & \lambda_j^n & \ddots & \ddots \\
\vdots & \vdots & \ddots & \ddots & \ddots \\
0 & 0 & 0 & \dots & \lambda_j^n
\end{matrix} \right),
\end{equation}
holds, we obtain that if $\lambda$ is an eigenvalue of $T$, then either $|\lambda| < 1$, or $|\lambda| = 1$ and the size of any corresponding Jordan-block is exactly $1\times 1$. 

On the other hand, if $J = U\oplus B$, then it is obviously powerbounded, and hence $T\in\PWB_d$.
\end{proof}

The forthcoming theorem says that the Ces\`aro asymptotic limit exists for $T\in\irM_d$ if and only if $T$ is powerbounded.

\begin{theorem}\label{C^d_Cer_mean_thm}
Let $T\in\PWB_d$ and by using the notations of Proposition \ref{J_dec_pwb_mx_prop} let us consider the matrices $J' = U\oplus 0$ and $T' = SJ'S^{-1}$. Then the Ces\`aro asymptotic limits $A_{T,C}$ and $A_{T',C}$ always exist and 
\[ A_{T,C} = A_{T',C}. \]
Conversely, if the Ces\`aro asymptotic limit exists for a matrix, then it is necessarily powerbounded.
\end{theorem}

\begin{proof}
For the first part, let us consider the following:
\begin{equation}\label{A_TC_and_AT'C_eq}
\begin{gathered}
 \frac{1}{n}\sum_{j=1}^n T^{*j}T^j - \frac{1}{n}\sum_{j=1}^n T'^{*j}T'^j \\
 = \frac{1}{n}S^{*-1}\Big(\sum_{j=1}^n J^{*j}S^*SJ^j - \sum_{j=1}^n J'^{*j}S^*SJ'^j\Big)S^{-1}  \\
 = S^{*-1}\Big(\frac{1}{n}\sum_{j=1}^n (0\oplus B)^{*j}S^*S(0\oplus B)^j + \frac{1}{n}\sum_{j=1}^n (U\oplus 0)^{*j}S^*S(0\oplus B)^j \\
 + \frac{1}{n}\sum_{j=1}^n (B\oplus 0)^{*j}S^*S(U\oplus 0)^j \Big)S^{-1}.
\end{gathered}
\end{equation}
Since $\lim_{j\to\infty} (0\oplus B)^{j} = 0$, we obtain that $A_{T',C}$ and $A_{T,C}$ co-exist simultaneously and if they exist, they have to be equal.

Now we consider the partial sums of the Ces\`aro asymptotic limit of $T'$: $\frac{1}{n}\sum_{j=1}^n T'^{*j}T'^j$. It is easy to see that multiplying from the right by a diagonal matrix acts as multiplication of the columns by the corresponding diagonal elements. Similarly, for the multiplication from the left action this holds with the rows. We have the next equality:
\begin{equation}\label{limit_eq}
\begin{gathered}
J'^{*j}S^*SJ'^j = \\
\left(\begin{matrix}
\|S e_1\|^2 & \lambda_2^n\overline{\lambda_1}^n \langle S e_1, S e_2\rangle & \dots & \lambda_m^n\overline{\lambda_1}^n \langle S e_1, S e_m\rangle & 0 & \dots & 0\\
\lambda_1^n\overline{\lambda_2}^n \langle S e_2, S e_1\rangle & \|S e_2\|^2 & \dots & \lambda_m^n\overline{\lambda_2}^n \langle S e_2, S e_m\rangle & 0 & \dots & 0 \\
\vdots & \vdots & \ddots & \vdots & \vdots & & \vdots  \\
\lambda_1^n\overline{\lambda_m}^n \langle S e_m, S e_1\rangle & \lambda_2^n\overline{\lambda_m}^n \langle S e_m, S e_2\rangle & \dots & \|S e_m\|^2 & 0 & \dots & 0 \\
0 & 0 & \dots & 0 & 0 & \dots & 0\\
\vdots & \vdots & & \vdots & \vdots & \ddots & \vdots \\
0 & 0 & \dots & 0 & 0 & \dots & 0\\
\end{matrix}\right).
\end{gathered}
\end{equation}
Since 
\[ \lim_{n\to\infty} \Bigg|\frac{1}{n}\sum_{j=1}^n \lambda^j\Bigg| = \lim_{n\to\infty} \frac{|\lambda^n-1|}{n|\lambda-1|} = 0 \]
if $|\lambda|\leq 1, \lambda \neq 1$ (for $\lambda = 1$ the above limit is 1) and multiplying by a fix matrix does not have an effect on the fact of convergence, we can easily infer that $A_{T',C}$ and hence $A_{T,C}$ exists.

For the reverse implication, let us assume that $A_{T,C}$ exists for a $T\in\irM_d$. According to (\ref{Cesaro_as_beh_eq}) there exists a large enough $\widetilde{M}>0$ such that
\[ \frac{1}{n}\sum_{j=1}^n \|T^j h\|^2 = \frac{1}{n}\sum_{j=1}^n \|SJ^jS^{-1} h\|^2 \leq \widetilde{M} \]
holds for each unit vector $h\in\C^d$. Since $S$ is bounded from below, the above inequality holds exactly when 
\[ \frac{1}{n}\sum_{j=1}^n \|J^j h\|^2 \leq M \]
holds for every unit vector $h\in\C^d$ with a large enough bound $M>0$. On the one hand, this implies that $r(J) = r(T) \leq 1$. On the other hand, if there is an at least $2\times 2$ $\lambda$-Jordan block in $J$ where $|\lambda| = 1$, then this above inequality obviously cannot hold for any unit vector (see (\ref{Jordan-block_power_eq})). This ensures the powerboundedness of $T$.
\end{proof}

Now we are in position to prove our first main theorem. Before that let us point out that if a sequence of matrices $\{S_n\}_{n=1}^\infty$ is entry-wise $L$-convergent then 
\[ \Llim_{n\to\infty} XS_n = X\cdot\Llim_{n\to\infty} S_n \] 
and 
\[ \Llim_{n\to\infty} S_nX = (\Llim_{n\to\infty} S_n)\cdot X \] 
hold. This can be easily verified from the linearity of $L$. 

\begin{proof}[Proof of Theorem \ref{C^d_B_lim_Cer_thm}]
The equality $\Llim_{n\to\infty} \lambda^n = \lambda\Llim_{n\to\infty} \lambda^n$ holds for every $|\lambda|\leq 1$ which gives us $\Llim_{n\to\infty} \lambda^n = 0$ for all $\lambda\neq 1, |\lambda|\leq 1$ (the Banach limit is trivially 1 if $\lambda = 1$). Now, if we take a look at equation (\ref{limit_eq}), we can see that $A_{T',L} = A_{T',C}$ holds for every Banach limit. Since 
\[ (0\oplus B)^{*j}S^*S(0\oplus B)^j + (0\oplus U)^{*j}S^*S(0\oplus B)^j + (B\oplus 0)^{*j}S^*S(0\oplus U)^j \to 0 \]
the equation-chain
\[ A_{T,L} = A_{T',L} = A_{T',C} = A_{T,C} \] 
is yielded.
\end{proof}

A natural question arises here. When does the sequence $\{T^{*n}T^n\}_{n=1}^\infty\subseteq\irM_d$ converge? The last theorem of the section is dealing with this question where we will use a theorem of G. Corach and A. Maestripieri. The symbol $\rank(A)$ denotes the rank of the matrix $A$.

\begin{theorem}\label{AT_char_thm}
The following are equivalent for a $T\in \PWB_d$
\begin{itemize}
\item[\textup{(i)}] the sequence $\{T^{*n}T^n\}_{n=1}^\infty\subseteq\irM_d$ converge,
\item[\textup{(ii)}] the eigenspaces of $T$ corresponding to eigenvalues with modulus 1 are mutually orthogonal to each other.
\end{itemize}
Moreover the following three sets coincide
\begin{equation}\label{sets_eq}
\begin{gathered}
\left\{ A\in\irB_+(\C^d) \colon \exists\; T\in\PWB_d, A = \lim_{n\to\infty} T^{*n}T^n \right\}, \\
\left\{ P^*P\in\irB_+(\C^d) \colon P\in\irM_d, P^2 = P \right\}, \\
\left\{ A\in\irB_+(\C^d) \colon \sigma(A)\subseteq \{0\}\cup[1,\infty), \dim\ker(A)\geq \rank E\big((1,\infty)\big) \right\},
\end{gathered}
\end{equation}
where $E$ denotes the spectral measure of $A$.
\end{theorem}

\begin{proof}
The (ii)$\Longrightarrow$(i) implication is quite straightforward from (\ref{limit_eq}).

For the reverse direction let us consider (\ref{limit_eq}) again. This tells us that if $\lambda_l \neq \lambda_k$, then $\langle S e_l, S e_k\rangle$ has to be 0 which means exactly the orthogonality.

In order to prove the further statement, we just have to take such a $T$ that satisfies (ii). Since the limit of $\{T^{*n}T^n\}_{n=1}^\infty$ exists if and only if the sequence $\{T'^{*n}T'^{n}\}_{n=1}^\infty$ converge (by a similar reasoning to the beginning of the proof of Theorem \ref{C^d_Cer_mean_thm}), we consider 
\[ T'^{*n}T'^n = S^{*-1}(J'^{*n}S^*SJ'^n)S^{-1}. \] 
Trivially, we can take an orthonormal basis in every eigenspace as columnvectors of $S$ and if we do so, then the columnvectors of $S$ corresponding to modulus 1 eigenvalues will form an orthonormal sequence, and the other columnvectors (i. e. the zero eigenvectors) will form an orthonormal sequence, too. But the whole system may fail to be an orthonormal basis. This implies that
\[ \lim_{n\to\infty} T'^{*n}T'^n = S^{*-1}(I\oplus 0)S^{-1} = S^{*-1}(I\oplus 0)S^*S(I\oplus 0)S^{-1}. \]
By Theorem 6.1 of \cite{CM} we get (\ref{sets_eq}).
\end{proof}

\section{The characterisation}\label{char_sec}

In the present section we will prove Theorem \ref{C_11_char_thm}, \ref{non-C11_char_thm} and \ref{2x2_regularity_thm}.

\begin{proof}[Proof of Theorem \ref{C_11_char_thm}] The $(i)\iff(ii)$ part follows from the results of the previous section. We begin with the (i)$\iff$(iv) part. Let us suppose that $T = SUS^{-1}$ holds with an $S\in GL_d$ and a $U = \diag(\lambda_1,\dots\lambda_d)\in \irU_d$. Of course, it can be supposed without loss of generality that $S$ has unit columnvectors (this is just a right choice of eigenvectors). Moreover, if an eigenvalue $\lambda$ has multiplicity more than one, then the corresponding unit eigenvectors (as columnvectors in $S$) can be chosen to form an orthonormal sequence in that eigenspace. Trivially, this does not change $T$. Now considering (\ref{limit_eq}), we get 
\[ \frac{1}{n} \sum_{j=1}^n U^{*j}S^*SU^j \to I \] 
and therefore 
\[ A_{T,C} = S^{*-1}S^{-1}. \]

In order to yield the other implication take an $S$ such that it has unit columnvectors. If we put $\lambda_j$-s to be pairwise different in $U$ and $T = SUS^{-1}$ then we obviously get $A_{T,C} = S^{*-1}S^{-1}$ from equation (\ref{limit_eq}).

\medskip

After that we turn to the (iv)$\iff$(iii) part. By the spectral mapping theorem we have $d = \tr(S^*S) = \tr(SS^*) = \sum_{j=1}^d \frac{1}{t_j}$, where $\tr(\cdot)$ denotes the trace.

In order to show the reverse direction, it would be enough to find such a unitary matrix $U\in\C^{d\times d}$ which satisfies 
\[ \big\|\diag(\sqrt{1/t_1},\dots,\sqrt{1/t_d})\cdot U e_j\big\| = 1. \] 
Indeed, if we chose 
\[S := \diag(\sqrt{1/t_1},\dots,\sqrt{1/t_d})\cdot U,\] 
$SS^*$ would become $\diag(1/t_1,\dots,1/t_d)$ and (\ref{unitary_eqv_eq}) would give what we want. 

The idea is that we put such complex numbers in the entries of $U$ which have modulus $1/\sqrt{d}$, because then the columnvectors of $S$ will be unit ones and we only have to be careful with the orthogonality of the columnvectors of $U$. In fact, the right choice is to consider a constant multiple of a Vandermonde matrix:
\[ U := \Big(\varepsilon^{(j-1)(k-1)}/\sqrt{d}\Big)_{j,k=1}^d \] 
where $\varepsilon = e^{2i\pi/d}$. We show that its columnvectors are orthogonal to each other which will complete the proof. For this we consider $j_1\neq j_2, j_1,j_2\in\{1,2,\dots d\}$ and the inner product
\[ \langle Ue_{j_1}, Ue_{j_1} \rangle = \sum_{k=1}^d \frac{\varepsilon^{(j_1-1)(k-1)}}{\sqrt{d}} \frac{\overline{\varepsilon^{(j_2-1)(k-1)}}}{\sqrt{d}} = \frac{1}{d} \sum_{k=1}^d \varepsilon^{(j_1-j_2)(k-1)} = 0. \]
This shows that $U$ is unitary.
\end{proof}

Now we are able to deal with the case when the stable subspace is non-trivial but first we have to start with a structure theorem of powerbounded matrices. In \cite[Lemma 1.]{Ke_isom_as} L. K\'erchy proved a generalization of \cite[Theorem II.4.1.]{NFBK}. This lemma can be re-phrased as in the next lemma. We say that a powerbounded operator $T$ is of class $C_{0\cdot}$ if $A_{T,L} = 0$, and it is of class $C_{1\cdot}$ if $A_{T,L}$ is injective. We remind the reader that $\irH_0(T)$ is hyperinvariant for $T$.

\begin{lemma}[K\'erchy]\label{Ker_lem}
Let us consider an arbitrary $T\in\PWB(\irH)$ with stable subspace $\irH_0$ and its block-matrix representation with respect to the orthogonal decomposition $\irH_0\oplus\irH_0^\perp$
\[ T = \left(\begin{matrix}
T_0 & R \\
0 & T_1
\end{matrix}\right). \]
Then $T_0\in C_{0\cdot}(\irH_0)$ and $T_1\in C_{1\cdot}(\irH_0^\perp)$.
\end{lemma}

It is worth noting that usually if we consider such a block upper triangular matrix above with an arbitrary element $R$, then usually $T$ is not powerbounded. However, the powerboundedness is automatic if we assume that $T$ is a matrix.

\begin{lemma}\label{structure_lem}
Suppose that the matrix $T\in\irM_d$ has the block-matrix upper triangular representation above, with respect to an orthogonal decomposition $\irH = \irH'\oplus\irH''$ such that $T_0\in C_{0\cdot}(\irH')$ ($R\in\irB(\irH'',\irH')$ is arbitrary) and $T_1\in C_{1\cdot}(\irH'')$ holds. Then necessarily $T$ is powerbounded and its stable subspace is precisely $\irH'$.
\end{lemma}

\begin{proof}
By an easy calculation we get
\[ T^n = \left(\begin{matrix}
T_0^n & R_n \\
0 & T_1^n
\end{matrix}\right) \]
where 
\[ R_n = \sum_{j=1}^{n} T_0^{n-j}RT_1^{j}. \] 
Let us assume that $\|T_0^n\|,\|T_1^n\|,\|R\| < M$ for each $n\in\N$. In order to see the powerboundedness of $T$, one just has to use that $\|T_0^n\|\leq r^n$ holds for large $n$-s with a number $r<1$, since we are in a finite dimensional space. It is quite straightforward that $\irH' \subseteq \irH_0(T)$ and since for any vector $x\in\irH''$ the sequence $\{T_1^n x\}_{n=1}^\infty$ does not converge to 0, then neither does $\{T^n x\}_{n=1}^\infty$. Consequently we obtain that $\irH' = \irH_0(T)$.
\end{proof}

The above proof works for those kind of infinite dimensional operators as well, for which $r(T_0) < 1$ is satisfied.

Now, we are in position to show the characterisation in the non-$C_{11}$ case.

\begin{proof}[Proof of Theorem \ref{non-C11_char_thm}] The equivalence of (i), (ii) and (iv) can be handled very similarly as in Theorem \ref{C_11_char_thm}.

\medskip

\emph{The (i)$\iff$(iii) part:} Let us write
\[ T^* = \left(\begin{matrix}
0 & \tilde{R} \\
0 & E^*
\end{matrix}\right) = \left(\begin{matrix}
0 & RE^* \\
0 & E^*
\end{matrix}\right) \quad (E\in C_{11}(\C^k)) \]
for the adjoint of a powerbounded matrix (up to unitarily equivalence) which is general enough for our purposes (see Theorem \ref{C^d_Cer_mean_thm} and Lemma \ref{structure_lem}). Since $E^*$ is invertible it is equivalent to investigate the two forms above. From the equation
\[ T^{*n}T^n = \left(\begin{matrix}
0 & RE^{*n} \\
0 & E^{*n}
\end{matrix}\right)
\left(\begin{matrix}
0 & 0 \\
E^nR^* & E^n
\end{matrix}\right) =
\left(\begin{matrix}
RE^{*n}E^nR^* & RE^{*n}E^n \\
E^{*n}E^nR^* & E^{*n}E^n
\end{matrix}\right),
\]
we get
\[ A_{T,C} = \left(\begin{matrix}
RA_{E,C}R^* & RA_{E,C} \\
A_{E,C}R^* & A_{E,C}
\end{matrix}\right). \]

Calculating the null-space of $A_{T,C}$ suggests that for the block-diagonalization we have to take the following invertible matrix
\[ X = 
\left(\begin{matrix}
I_l & R \\
-R^* & I_k
\end{matrix}\right)
\left(\begin{matrix}
(I_l + RR^*)^{-1/2} & 0 \\
0 & (I_k + R^*R)^{-1/2}
\end{matrix}\right). \]
The inverse of $X$ is
\[ X^{-1} = \left(\begin{matrix}
(I_l + RR^*)^{-1/2} & 0 \\
0 & (I_k + R^*R)^{-1/2}
\end{matrix}\right)
\left(\begin{matrix}
I_l & -R \\
R^* & I_k
\end{matrix}\right) \]
which shows that $X\in\irU_d$. With a straightforward calculation we derive
\[ X^{-1}A_{T,C}X = 0_l \oplus \big[(I_k + R^*R)^{1/2}A_{E,C}(I_k + R^*R)^{1/2}\big]. \]
Here $(I_k + R^*R)^{1/2} = I_k + Q$ holds with a positive semi-definite $Q$ for which $\rank(Q)\leq l$ holds and conversely, every such $I_k+Q$ can be given in the form $(I_k + R^*R)^{1/2}$. So the set of all $L$-asymptotic limits of $l$-stable powerbounded matrices (again, up to unitarily equivalence) are given by:
\[ 0_l \oplus \left[(I_k + Q)A_{E,C}(I_k + Q)\right] \quad (Q\in\irB_+(\C^k), \rank(Q) < l, E\in C_{11}(\C^k)) \]
and every positive operator of such form is the $L$-asymptotic limit of an $l$-stable powerbounded matrix.

Finally, let us write
\[ (I_k+Q)^{-2} = U\cdot\diag(q_1,\dots q_k)\cdot U^* \]
and 
\[ U^*A_{E,C}^{-1}U = (\alpha_{i,j})_{i,j=1}^k, \]
where $U \in \irU_k$ and $q_j\leq 1$ for each $1\leq j\leq k$. Therefore
\[ \frac{1}{t_1} + \dots + \frac{1}{t_k} = \tr\big((I_k+Q)^{-1}A_{E,C}^{-1}(I_k+Q)^{-1}\big) = \tr\big(A_{E,C}^{-1}(I_k+Q)^{-2}\big) \]
\[ = \tr\big(U^*A_{E,C}^{-1}U\cdot\diag(q_1,\dots q_k)\big) = \sum_{j=1}^k q_j \alpha_{j,j} \leq \tr(A_{E,C}^{-1}) = k, \]
is fulfilled for each $l$-stable powerbounded matrix $T$. For the reverse, set some positive numbers $t_1,\dots t_k$ such that 
\[ \frac{1}{t_1} + \dots + \frac{1}{t_k} \leq k \]
is valid. If we take such a $0<c<1$ for which
\[ \frac{1}{c\cdot t_1} + \frac{1}{t_2} + \dots + \frac{1}{t_k} = k, \] 
then obviously $A = \diag(c\cdot t_1, t_2\dots t_k)$ arises as the Ces\`aro asymptotic limit of a powerbounded operator from $C_{11}(\C^k)$. Therefore, if we take the rank-one $Q = \diag(1/\sqrt{c}-1,0,\dots 0)$ positive semi-definite matrix, we get $(I_k+Q)A(I_k+Q) = \diag(t_1, t_2\dots t_k)$. This ends the proof.
\end{proof}

After the proof of the characterisations we give the proof of Theorem \ref{2x2_regularity_thm}. 

\begin{proof}[Proof of Theorem \ref{2x2_regularity_thm}]
Let $T = S\diag(\lambda_1,\lambda_2)S^{-1}$ with $|\lambda_1| = |\lambda_2| = 1$ and $S\in GL_2$ for which the columnvectors are unit ones. We have learned from the proof of Theorem \ref{C_11_char_thm} that in this case we have $A_{T,C}^{-1} = SS^*$. The matrix $S$ can be written in the following form
\[ S = \left(\begin{matrix}
\mu_{1,1}s & \mu_{1,2}t \\
\mu_{2,1}\sqrt{1-s^2} & \mu_{2,2}\sqrt{1-t^2}
\end{matrix}\right) \]
where $|\mu_{j,l}| = 1$ ($j,l\in\{1,2\}$) and $s,t\in[0,1]$ provided that the above matrix is invertible.

By taking $\diag(1/\mu_{1,1},1/\mu_{2,1})\cdot S$ instead of $S$ and using (\ref{unitary_eqv_eq}), we can see that without loss of generality $\mu_{1,1} = \mu_{2,1} = 1$ can be assumed, so we have
\[ S = \left(\begin{matrix}
s & \mu_{1,2}t \\
\sqrt{1-s^2} & \mu_{2,2}\sqrt{1-t^2}
\end{matrix}\right) \]
and thus
\[ A_{T,C}^{-1} = SS^* = \left(\begin{matrix}
s^2+t^2 & s\sqrt{1-s^2} + t\sqrt{1-t^2}\overline{\mu_{22}}\mu_{12} \\
s\sqrt{1-s^2} + t\sqrt{1-t^2}\mu_{22}\overline{\mu_{12}} & 2-s^2-t^2
\end{matrix}\right). \]

Finally, since $T^* = S^{*-1}\diag(\overline{\lambda_1},\overline{\lambda_2})S^*$ and we have 
\[ S^{*-1} = \frac{1}{s\sqrt{1-t^2}\overline{\mu_{22}} - t\sqrt{1-s^2}\overline{\mu_{12}}} \left(\begin{matrix}
\sqrt{1-t^2}\overline{\mu_{22}} & -\sqrt{1-s^2} \\
-t\overline{\mu_{12}} & s
\end{matrix}\right) \]
we immediately obtain
\[ A_{T^*,C}^{-1} = 
\left(\begin{matrix}
\sqrt{1-t^2}\overline{\mu_{22}} & -\sqrt{1-s^2} \\
-t\overline{\mu_{12}} & s
\end{matrix}\right)
\left(\begin{matrix}
\sqrt{1-t^2}\mu_{22} & -t\mu_{12} \\
-\sqrt{1-s^2} & s
\end{matrix}\right) \]
\[ = 
\left(\begin{matrix}
2-s^2-t^2 & -s\sqrt{1-s^2} - t\sqrt{1-t^2}\overline{\mu_{22}}\mu_{12} \\
-s\sqrt{1-s^2} - t\sqrt{1-t^2}\mu_{22}\overline{\mu_{12}} & s^2+t^2
\end{matrix}\right) \]
and this implies (\ref{2_dim_nice_eq}).
\end{proof}

In the proof we used that the inverse of a $2\times 2$ matrix can be expressed quite nicely. A natural question arises, what happens in higher dimension? If the dimension of $\irH$ is infinite then we can get counterexamples quite easily. Let us consider a weighted bilateral shift which has weights 1 everywhere except for one weight which is 1/2. This trivially defines a contraction and hence both $A_{T,C} = A_T$ and $A_{T^*,C} = A_{T^*}$ hold. Moreover, $A_T$ and $A_{T^*}$ are invertible which implies that $T$ is similar to a unitary. But $A_T, A_{T^*} \leq I$, $A_T \neq I$ and $A_{T^*} \neq I$ which implies $A_{T,C}^{-1} + A_{T^*,C}^{-1} \geq 2I$ and $A_{T,C}^{-1} + A_{T^*,C}^{-1} \neq 2I$. 

In the matrix case we do not have such an easy counterexample. However, computations tell us that usually (\ref{2_dim_nice_eq}) does not hold even for $3\times 3$ matrices. For example if we consider 
\[ T = \left(\begin{matrix}
i & 2 & 1 \\
0 & 1 & i \\
1 & 0 & 4
\end{matrix}\right)
\left(\begin{matrix}
1 & 0 & 0 \\
0 & -1 & 0 \\
0 & 0 & i
\end{matrix}\right)
\left(\begin{matrix}
i & 2 & 1 \\
0 & 1 & i \\
1 & 0 & 4
\end{matrix}\right)^{-1} \]
then we get that the eigenvalues of $A_{T,C}^{-1} + A_{T^*,C}^{-1}$ are approximately the numbers $1.27178, 2.1285$ and $2.59972$.

Of course, a natural question arises.

\begin{question} 
\textup{How can we describe the set of all pairs $\{(A_{T,C},A_{T^*,C})\colon T\in\PWB_d\}$ if $d\in\N$ is fixed?}
\end{question}

A weaker form of the question above is the following.

\begin{question} 
\textup{How can we describe the set of all pairs $\{(A_{T,C},A_{T^*,C})\colon T\in C_{11}(\C^d)\}$ if $d\in\N$ is fixed?}
\end{question}

As a consequence of Theorem \ref{2x2_regularity_thm} and the Corach-Maestripieri theorem (a 1-stable $2\times 2$ powerbounded matrix is always a projection) we could do these descriptions very easily in the case when $d = 2$, but because of its easiness, we omit it in this paper.

\section{Some properties in arbitrary dimension} \label{norm_of_as_lim_chap}

In the last section we prove Theorem \ref{A_TC_norm} and \ref{A_TL_norm} and we also give some examples.

\begin{proof}[Proof of Theorem \ref{A_TC_norm}]
Suppose that $0 < \|A_{T,C}\|$, we will show that in this case $1 \leq \|A_{T,C}\|$ holds as well. Consider a unit vector $v\in\irH$ and an arbitrarily small number $\varepsilon > 0$ for which
\[ \big\|\sqrt{A_{T,C}}\; v\big\| > \big(\big\|\sqrt{A_{T,C}}\big\|-\varepsilon\big) \]
is satisfied. By (\ref{Cesaro_iter_eq}):
\[ \Bigg\|\sqrt{A_{T,C}} \frac{T^k v}{\|T^k v\|}\Bigg\| = \frac{\|\sqrt{A_{T,C}} v\|}{\|T^k v\|} > \frac{\big(\big\|\sqrt{A_{T,C}}\big\|-\varepsilon\big)}{\|T^k v\|}. \]
Since $\liminf_{k\to\infty} \|T^k v\| \leq \|\sqrt{A_{T,C}}\|$, for every $\eta > 0$ there exists a $k_0\in\N$ for which $\|T^{k_0} v\| \leq \|\sqrt{A_{T,C}}\|+\eta$ holds. This implies that
\[ \Bigg\|\sqrt{A_{T,C}} \frac{T^{k_0} v}{\|T^{k_0} v\|}\Bigg\| > \frac{(\|\sqrt{A_{T,C}}\|-\varepsilon)}{\|T^{k_0} v\|} \geq \frac{(\|\sqrt{A_{T,C}}\|-\varepsilon)}{\|\sqrt{A_{T,C}}+\eta\|}. \]
But this is true for all $\eta > 0$, therefore
\[ \sqrt{\|A_{T,C}\|} = \big\|\sqrt{A_{T,C}}\big\| \geq \frac{\big(\big\|\sqrt{A_{T,C}}\big\|-\varepsilon\big)}{\big\|\sqrt{A_{T,C}}\big\|}. \]
Since $\varepsilon > 0$ was arbitrarily small, we conclude $\|A_{T,C}\| \geq 1$.
\end{proof}

The next proof is similar to the one above, but we note that the squares are needed, because the Banach limits are not multiplicative.

\begin{proof}[Proof of Theorem \ref{A_TL_norm}]
Assume that $0 < \|A_{T,L}\|$ happens. Take a vector $v\in\irH, \|v\| = 1$ and an arbitrarily small number $\varepsilon > 0$ such that 
\[ \Llim_{n\to\infty} \|T^n v\|^2 = \big\|\sqrt{A_{T,L}} v\big\|^2 > \big(\big\|\sqrt{A_{T,L}}\big\|-\varepsilon\big)^2 \]
is satisfied. Consider now the following inequality:
\[ \Bigg\|\sqrt{A_{T,L}} \frac{T^k v}{\|T^k v\|}\Bigg\|^2 = \frac{\|\sqrt{A_{T,L}} v\|^2}{\|T^k v\|^2} > \frac{\big(\big\|\sqrt{A_{T,L}}\big\|-\varepsilon\big)^2}{\|T^k v\|^2}. \]
Since $\liminf_{k\to\infty} \|T^k v\|^2 \leq \Llim_{k\to\infty} \|T^k v\|^2 \leq \|\sqrt{A_{T,L}}\|^2$, for every $\eta > 0$ there exists a $k_0\in\N$ for which $\|T^{k_0} v\|^2 \leq (\|\sqrt{A_{T,L}}\|+\eta)^2$ holds. This suggests that
\[ \Bigg\|\sqrt{A_{T,L}} \frac{T^{k_0} v}{\|T^{k_0} v\|}\Bigg\|^2 > \frac{(\|\sqrt{A_{T,L}}\|-\varepsilon)^2}{\|T^{k_0} v\|^2} \geq \frac{(\|\sqrt{A_{T,L}}\|-\varepsilon)^2}{\|\sqrt{A_{T,L}}+\eta\|^2}. \]
Since this holds for every $\eta > 0$, we infer that
\[ \|A_{T,L}\| = \big\|\sqrt{A_{T,L}}\big\|^2 \geq \frac{\big(\big\|\sqrt{A_{T,L}}\big\|-\varepsilon\big)^2}{\big\|\sqrt{A_{T,L}}\big\|^2}. \]
In the beginning we could choose an arbitrarily small $\varepsilon > 0$, hence $\|A_{T,L}\| \geq 1$ is verified.
\end{proof}

Now we are intend to give some examples. First we provide an example for such a powerbounded weighted shift operator $T$ for which the $L$-asymptotic limit really depends on the choice of the particular Banach limit $L$ and moreover, the Ces\`aro asymptotic limit of $T$ exists. We will use a characterisation from \cite{Lo} of all the possible Banach limits of a bounded real sequence. Namely, G. G. Lorentz proved that for every $\underline{x}\in\ell^\infty$ real sequence the equality
\[ \big\{\Llim_{n\to\infty} x_n \colon L\in\irB\big\} = [q'(\underline{x}),q(\underline{x})] \subseteq \R \]
holds where $\irB$ denotes the set of all Banach limits and
\[ q(\underline{x}) = \inf\Big\{\limsup_{k\to\infty} \frac{1}{p}\sum_{j=1}^{p} x_{n_j+k} \colon p\in\N, n_1,\dots n_p \in\N \Big\}, \]
\[ q'(\underline{x}) = \sup\Big\{\liminf_{k\to\infty} \frac{1}{p}\sum_{j=1}^{p} x_{n_j+k} \colon p\in\N, n_1,\dots n_p \in\N \Big\}. \]

\begin{example}
\textup{Let us fix an orthonormal basis $\{e_j\}_{j\in\N}$ and define $T$ in the following way}
\[ T e_j = \left\{ \begin{matrix}
\sqrt{2}\cdot e_{j+1} & j = 3^l \textup{ for some } l\in\N \\
\sqrt{1/2}\cdot e_{j+1} & j = 3^l+l \textup{ for some } l\in\N \\
e_{j+1} & \textup{otherwise}
\end{matrix}\right.. \]
\textup{Since $\|T^n\| = \sqrt{2}$ holds for each $n\in\N$, $T$ is powerbounded. A rather easy calculations shows that the equation}
\[ T^{*n}T^n e_1 = \left\{ \begin{matrix}
2 e_1 & 3^l \leq n < 3^l+l \textup{ for some } l\in\N \\
e_1 & \textup{otherwise}
\end{matrix}\right. \]
\textup{is valid. Therefore by Lorentz's characterisation we get}
\[ \big\{\langle A_{T,L} e_1, e_1\rangle \colon L\in\irB\big\} = [1,2]. \]

\textup{It is quite easy to see that the sequence $\{T^{*n}T^n\}_{n=1}^\infty$ consists of diagonal operators (with respect to our fixed orthonormal basis). Thus $A_{T,C}$ is also diagonal and the convergence $\frac{1}{n}\sum_{j=1}^n T^{*j}T^j \to A_{T,C}$ holds in the strong operator topology in case when $A_{T,C}$ exists. But the existence of the Ces\`aro limit can be verified by a very simple calculation, in fact we have $A_{T,C} = I$.}
\end{example}

\smallskip

In what follows we show two further examples concerning the existence of the Ces\`aro asymptotic limit and powerboundedness. It was mentioned in the Introduction that none of them implies the other one which will be illustrated bellow.

\begin{example}
\textup{We define $T$ with the equation}
\[ T e_j = \left\{ \begin{matrix}
\sqrt{2}\cdot e_{j+1} & j = 3^l \textup{ for some } l\in\N \\
\sqrt{1/2}\cdot e_{j+1} & j = 2\cdot 3^l \textup{ for some } l\in\N \\
e_{j+1} & \textup{otherwise}
\end{matrix}\right.. \]
\textup{Clearly}
\[ T^{*n}T^n e_1 = \left\{ \begin{matrix}
2 \cdot e_1 & 3^l \leq n < 2\cdot 3^l \textup{ for some } l\in\N \\
e_1 & \textup{otherwise}
\end{matrix}\right. \]
\textup{holds. The Ces\`aro means of the sequence $\{\langle T^{*n}T^n e_1,e_1\rangle\}_{n\in\N}$ does not converge, hence the the Ces\`aro asymptotic limit of $T$ does not exist. On the contrary, $\|T^n\| = \sqrt{2}$ is satisfied for every $n\in\N$, thus $T$ is powerbounded.}
\end{example}

\begin{example}
\textup{Our shift operator is defined as follows}
\[ T e_j = \left\{ \begin{matrix}
\sqrt{2}\cdot e_{j+1} & 3^l \leq j < 3^l + l \textup{ for some } l\in\N \\
e_{j+1} & \textup{otherwise}
\end{matrix}\right.. \]
\textup{Clearly it is not powerbounded, because $\|T^n\| = \sqrt{2}^n$. But on the other hand, it is not hard to see that the Ces\`aro means of the sequence $\{T^{*n}T^n\}_{n\in\N}$ of positive operators converge strongly to $I=A_{T,C}$.}
\end{example}

\bigskip

\textbf{Acknowledgement.} The author emphasises his thank to L. K\'erchy and L. Ozsv\'art for some useful suggestions.

\bibliographystyle{ams}

\bigskip

\noindent {\sc Bolyai Institute, University of Szeged, Aradi v\'ertan\'uk tere 1, H-6720, Szeged, Hungary} and

\noindent {\sc MTA-DE "Lend\"ulet" Functional Analysis Research Group, Institute of Mathematics, University of Debrecen, H-4010 Debrecen, P.O. Box 12, Hungary

\bigskip

\noindent E-mail address: gehergy@math.u-szeged.hu

\noindent URL: http://www.math.u-szeged.hu/~gehergy/}

\end{document}